\documentclass[12pt, oneside]{scrartcl}
\usepackage{graphicx} 
\usepackage{amssymb}
\usepackage{amsmath}
\usepackage{hyperref}
\usepackage{amsthm}
\usepackage{cleveref}
\usepackage{tikz}
\newtheorem{theorem}{Theorem}[section]
\newtheorem{lemma}{Lemma}[section]
\newtheorem{problem}{Problem}
\newtheorem{claim}{Claim}
\newtheorem{remark}{Remark}

\newtheorem{conjecture}{Conjecture}
\usepackage{float}

\usepackage{setspace}
\title{List Coloring of the Square of 4-Irregular Graphs}
\author{Sara Al Hajjar}
\date{}

\author{Sara Al Hajjar \footnote{
Kalma Laboratory, Lebanese University, Beirut, Lebanon \\
Univ. Bordeaux, CNRS, Bordeaux INP, LaBRI, UMR 5800, F-33400, Talence,
France
}}
\date{}
\begin{document}

\maketitle
\begin{abstract}
The square of a graph $G$ is the graph obtained from $G$ after adding an edge between any two vertices of distance $2$. A $k$-irregular graph is a graph with maximum degree $k$ such that vertices of degree $k$ are not adjacent.
A \textit{list assignment} of a graph is a function $L$ that assigns to each vertex a list of permissible colors. 
The graph is said to be \textit{$L$-colorable} if there exists a proper coloring $f$ such that $f(v) \in L(v)$ for every vertex $v$. A graph $G$ is called \textit{$k$-choosable} if it is $L$-colorable for every list assignment where each list has exactly $k$ colors. The \textit{list chromatic number} of $G$, denoted by $\chi_l(G)$, is the smallest integer $k$ for which $G$ is $k$-choosable. Cranston and Kim \cite{ck} showed that $\chi_l(G^2) \leq 8$ for all subcubic graphs except the Petersen Graph. Moreover, Cranston and Kim \cite{ck} conjectured that for graphs with maximum degree $k$ and maximum clique size $w (G^2)\leq k^2-1$, we have $\chi_l(G^2) \leq k^2-1$. We prove that for a 4-irregular graph $G$, we have $\chi_l(G^2) \leq 11$. Moreover, we provide an example to show that this bound is sharp.
\end{abstract}

\section{Introduction}
\begin{large}
The discussion in this paper is restricted to finite simple graphs, with the usual notation conventions followed. In a graph $G$, the set of neighbors of a vertex $v$ is denoted by $N_G(v)$. The degree of a vertex $v$, written as $d_G(v)$, represents the number of its neighbors. 
For simplicity, we write $N(v)$ (resp., $d(v)$) instead of $N_G(v)$ (resp., $d_G(v)$). 
The maximum degree and minimum degree of $G$ are denoted by $\Delta(G)$ and $\delta(G)$, respectively. The maximum average degree of a graph $G$, denoted by $mad(G)$, is defined as
$ mad(G)=\max_{H\subsetneq G}\frac{2|E(H)|}{|V(H)|}$,
where the maximum is taken over all subgraphs $H$ of $G$. We say a graph $G$ is a $k$-irregular graph if $\Delta(G)=k$ and the vertices of degree $k$ in $G$ are not adjacent. A vertex of degree $k$ is called a $k$-vertex. We say $u$ is a $k$-neighbor of a vertex $v$ if $u$ is a $k$-vertex adjacent to $v$. The girth of a graph $G$ is the length of its minimal cycle and is denoted by $g(G)$. A $k$-cycle is a cycle of length $k$. Let $C$ be a cycle in $G$. A chord is an edge that joins two non-adjacent vertices in $C$. A clique in a graph $G$ is a subset of $V(G)$ that induces a complete subgraph of $G$. The clique number of $G$, denoted by $w(G)$, is the maximum order of a clique in $G$.

A list assignment of a graph $G$ is a function $L$ that assigns to each vertex $v \in V(G)$ a list $L(v)$ of available colors. 
A graph $G$ is said to be $L$-colorable if there exists a proper coloring $f$ such that $f(v) \in L(v)$ for every vertex $v$, and $f(u) \neq f(v)$ whenever $uv \in E(G)$. 
If $G$ is $L$-colorable for every list assignment $L$ in which each list $L(v)$ contains exactly $k$ colors, then $G$ is called $k$-choosable. 
The smallest integer $k$ for which $G$ is $k$-choosable is called the list chromatic number of $G$, denoted by $\chi_l(G)$. We are interested in the list chromatic number of $G^2$ when $G$ is a $4$-irregular graph.

\begin{figure}[h!]
\centering
\begin{tikzpicture}[scale=1.8, every node/.style={circle, fill=black, inner sep=1.8pt}]

  \node (x1) at (90:1.5) [label=above:$x_1$] {};
  \node (x2) at (18:1.5) [label=right:$x_2$] {};
  \node (x3) at (-54:1.5) [label=below right:$x_3$] {};
  \node (x4) at (-126:1.5) [label=below left:$x_4$] {};
  \node (x5) at (162:1.5) [label=left:$x_5$] {};

  \node (x6) at (90:0.7) [label=above:$x_6$] {};
  \node (x7) at (18:0.7) [label=right:$x_7$] {};
  \node (x8) at (-54:0.7) [label=below right:$x_8$] {};
  \node (x9) at (-126:0.7) [label=below left:$x_9$] {};
  \node (x10) at (162:0.7) [label=left:$x_{10}$] {};

  \node (v) at (0,-2.0) [label=below:$v$] {};

  \draw (x1)--(x2)--(x3)--(x4)--(x5)--(x1);

  \draw (x6)--(x8)--(x10)--(x7)--(x9)--(x6);

  \draw (x1)--(x6);
  \draw (x2)--(x7);
  \draw (x3)--(x8);
  \draw (x4)--(x9);
  \draw (x5)--(x10);

  \draw (v) .. controls (-3.5,-0.2) .. (x1);

  \draw (v)--(x8);
  \draw (v)--(x9);

\end{tikzpicture}

\caption{A $4$-irregular graph with $G^2=K_{11}$}
\label{fig:petersen_with_v}
\end{figure}
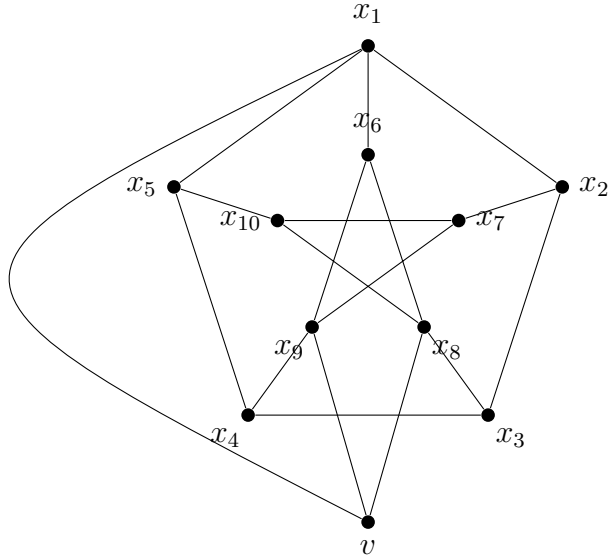

A $2$-distance $k$-coloring of a graph $G$ is a mapping $\phi : V(G) \to \{1,2,\ldots,k\}$ such that $\phi(v_1) \neq \phi(v_2)$ whenever $d(v_1,v_2) \leq 2$, where $v_1$ and $v_2$ are any two vertices in $G$.  
The $2$-distance chromatic number of a graph $G$, denoted by $\chi_2(G)$, is the minimum integer $k$ such that $G$ admits a $2$-distance $k$-coloring.
\\ 

Several papers have studied Wegner's conjecture \cite{11} regarding the  $2$-distance chromatic number of planar graphs. Wegner conjectured the following:
\begin{conjecture}(Wegner \cite{11}) If $G$  is a planar graph with maximum degree $\Delta$, then 
$\chi_2(G) \leq 7$ if $\Delta=3$, $\chi_{2} (G)  \leq\Delta + 5$ if $4 \leq \Delta \leq 7$ and $\chi_{2}(G)\leq \lfloor\frac{3\Delta}{2}\rfloor+1$ if $\Delta \ge 8$. \end{conjecture} 
This conjecture remains largely open. Thomassen~\cite{9} proved it for planar graphs with maximum degree $\Delta=3$. For larger values of $\Delta$, several upper bounds on $\chi_2(G)$ have been established. In particular, Agnarsson and Halldórsson~\cite{1} proved that
$\chi_2(G)\le \left\lfloor\frac{9\Delta}{5}\right\rfloor+2$
for planar graphs with $\Delta\ge749$, and this was later improved by Borodin \emph{et al.}~\cite{2} to
$\chi_2(G)\le \left\lceil\frac{9\Delta}{5}\right\rceil+1 $
for $\Delta\ge47$. More recently, Bousquet~\cite{3} proved that $\chi_2(G)\le2\Delta+7$ for $\Delta\ge9$ and $\chi_2(G)\le21$ for $\Delta\le6$. For planar graphs with $\Delta\le5$, Zhu and Bu~\cite{12} proved that $\chi_2(G)\le20$, and this bound was subsequently reduced to $19$, $18$, $17$, and finally $16$ by Chen~\cite{4}, Aoki~\cite{6}, Zou \emph{et al.}~\cite{13}, and Zakir Deniz~\cite{5}, respectively. Further results on the $2$-distance chromatic number of planar graphs can be found in \cite{7,8,10,12}.

\begin{conjecture} (Cranston \cite{ck})
 Let $G$ be a graph with maximum degree $k$ and $w(G^2)\leq k^2-1$. Then, we have $\chi_l(G^2) \leq k^2-1$.
    
\end{conjecture}

Cranston and Kim \cite{ck} showed that $\chi_l(G^2) \leq 8$ for all subcubic graphs except the Petersen Graph. Dolama and
Sopena \cite{ds} showed that $\chi_l(G^2)=\Delta +1$ when $\Delta(G)\ge 4$ and $mad(G)\leq \frac{16}{7}$. Cranston and \v{S}krekovski \cite{cs} proved that
$\chi_\ell^2(G)=\Delta+1$ whenever $\Delta\ge5$ and
$mad(G)<2+\frac{4\Delta-8}{5\Delta+2}$.

Let $G$ be a planar graph with $\Delta=3$ and girth $g$. Cranston and Kim \cite{ck} showed
that $\chi_{l}(G^2)\le 7$ if $g\ge 7$. More recently, Kim and Lian \cite{kl} showed that
$\chi_{l}(G^2)\le 7$ if $g\ge 6$. The same bound was proved by Jin, Kang,
and Kim \cite{jkk} for graphs with no $4$-cycles and no $5$-cycles; and by Kim, Lian,
Nakamoto, and Ozeki \cite{klno} for graphs with no $5$-cycles.
\\

We will prove that if $G$ is a $4$-irregular graph, then $\chi_l(G^2) \leq 11$. This bound is the best possible as illustrated in Figure \ref{fig:petersen_with_v}.
\\

The proof proceeds by assuming the existence of a minimal counterexample and deriving increasingly restrictive structural properties. We first exclude several short configurations, showing in particular that the graph has a girth at least $7$. These structural restrictions allow us to focus on a shortest cycle and its neighborhood. The final argument combines the obtained structural properties with carefully chosen list-coloring extensions to the whole graph, hence proving that the minimal counterexample graph does not exist, thereby completing the proof. 
\\

Throughout the paper, we partially color $G^2$ using lists of size 11. To extend the list-coloring to the whole graph $G^2$, we repeatedly use the notion of \emph{saving a color} at a vertex $v$. By this, we mean coloring two neighbors of $v$ in $G^2$ while reducing the number of available colors at $v$ by only one. A typical situation arises when $v$ is adjacent in $G^2$ to two nonadjacent vertices $v_1$ and $v_2$, and $|L(v_1)|+|L(v_2)|>|L(v)|$. This inequality guarantees either that $L(v_1)\cap L(v_2)\neq\emptyset$ or that $(L(v_1)\cup L(v_2))\setminus L(v)\neq\emptyset$. In the former case, we assign a common color to $v_1$ and $v_2$. In the latter, we color one of $v_1$ or $v_2$ with a color belonging to $(L(v_1)\cup L(v_2))\setminus L(v)$ and color the other arbitrarily. In either case, only one color is removed from the list of $v$.

\section{Main Result}
\begin{theorem} \label{thm1} 
If $G$ is a $4$-irregular graph, then $\chi_l(G^2) \leq 11$. \end{theorem}
The bound of \Cref{thm1} is sharp since we can find a $4$-irregular graph (see Fig. \ref{fig:petersen_with_v}) obtained from the Petersen Graph after adding a vertex and 3 edges such that $G^2=K_{11}$
and hence we have $\chi_l(G^2)=11$.
\\ 

Assume, for a contradiction, that \Cref{thm1} fails and let $G$ be a counterexample minimizing $|E(G)|$. Let $L$ be a list of size $11$. We will show that in fact we have $\chi_l(G^2) \leq 11$, a contradiction. So, $G$ does not exist and \Cref{thm1} holds. \\

Note that we always have $\chi_l(H^2) \leq 11$ for $H=G \backslash\{e\}$ for any edge $e\in E(G)$ by minimality of $|E(G)|$. Thus, to extend the list-coloring to the whole graph $G^2$, we only need to recolor the vertices of the edge $e$. Thus, we will always assume that by minimality of $|E(G)|$ we can color all the vertices of $G^2$ using lists of size 11 except two or more adjacent vertices.

\subsection{A Structural Analysis of G}
 Consider a list of colors of size 11 to each vertex $v$ in $G$, denoted by $L(v)$. 
 
 \begin{remark}Suppose that $G^2$ is partially colored using the lists of size $11$. Let $v$ be an uncolored vertex in $G^2$. Then, $|L(v)|\ge 11-k$ where $k$ is the number of colored neighbors of $v$ in $G^2$. 
 \end{remark}
\begin{remark}
    Suppose that $G^2$ is partially colored using the lists of size $11$. Let $v$ be an uncolored vertex in $G^2$. If $v$ has at least two uncolored neighbors in $G^2$, then $v$ has an available color.
\end{remark}

\begin{lemma} \label{l1} $\delta(G) \ge 3$. \end{lemma}
\begin{proof} Suppose, to the contrary, that $G$ has a vertex $v$ such that $d(v) \leq 2$. Let $v_1\in N(v)$, and let  $H=G-\{vv_1\}$. By minimality  of $|E(G)|$, we have $\chi_l(H^2) \leq 11$. Now, we uncolor $v$ and $v_1$ in $G$. We have $|L(v_1)| \ge 1$ and $|L(v)| \ge 4.$ Thus, we can greedily color $v_1$ then $v$ in $G^2$ to get $\chi_l(G^2) \leq 11$, a contradiction. \end{proof}

\begin{lemma} \label{l2} A $3$-vertex has at most one neighbor of degree $3$. \end{lemma}
\begin{proof} Suppose, to the contrary, that $G$ has a $3$-vertex $v$ adjacent to two $3$-vertices. Let $v_1\in N(v)$ such that $d(v_1)=3$. By minimality of $|E(G)|$, we can greedily color all the vertices in $G^2$ except $v$ and $v_1$. We have $|L(v)| \ge 2$ and $
|L(v_1)|\ge 1$. Thus, we greedily color $v_1$ then $v$ in $G^2$ to get $\chi_l(G^2) \leq 11$, a contradiction. \end{proof}
\begin{lemma} \label{l3} $g(G) \ge 4$. 
\end{lemma}
\begin{proof} Suppose that $G$ has a 3-cycle $C$. Let $v_1v_2v_3$ be this cycle. Since $G$ is $4$-irregular, at most one vertex of $C$ is of degree $4$. By \Cref{l2}, we deduce that exactly one vertex is of degree $4$. By minimality of $|E(G)|$, we can greedily color all the vertices in $G^2$ except $v_1$, $v_2$, and $v_3$. Since  $|L(v_i)| \ge 3$ for $i=1,2,3$, we can greedily color $v_1$, $v_2$, and $v_3$ in $G^2$ to obtain $\chi_l(G^2) \leq 11$, a contradiction. \end{proof}
\begin{lemma} \label{l5} If a $3$-vertex is contained in a $4$-cycle, then all its neighbors are of degree $4$. \end{lemma}
\begin{proof} Suppose that $G$ has a $3$-vertex $v$ contained in a $4$-cycle and having a $3$-neighbor $v_1$. By minimality of $|E(G)|$, we can greedily color all the vertices in $G^2$ except $v$ and $v_1$. We have $|L(v)| \ge 2$ and $
|L(v_1)|\ge 1$. Thus, we greedily color $v_1$ then $v$ in $G^2$ to get $\chi_l(G^2) \leq 11$, a contradiction.   \end{proof}

\begin{lemma} \label{l4} A vertex $v$ is contained in at most one $4$-cycle. \end{lemma}
\begin{proof} Suppose that $G$ has a vertex $v$ contained in two $4$-cycles. Then, there exists $v_1\in N(v)$ such that $v_1$ is also contained in two $4$-cycles. Since $G$ is $4$-irregular, one of $v$ and $v_1$ is a $3$-vertex. By \Cref{l5}, we deduce that one of $v$ and $v_1$ is a $4$-vertex. By minimality of $|E(G)|$, we can greedily color all the vertices in $G^2$ except $v$ and $v_1$. We have $|L(v)| \ge 2$ and $
|L(v_1)|\ge 2$. Thus, we greedily color $v_1$ and $v$ in $G^2$ to get $\chi_l(G^2) \leq 11$, a contradiction. \end{proof}

\begin{lemma} \label{l6} $g(G) \ge 5$. \end{lemma}
\begin{proof} Suppose that $G$ has a $4$-cycle $C=v_1v_2v_3v_4$. Since $G$ is $4$-irregular, we deduce that $C$ contains at least two vertices of degree $3$. By \Cref{l5}, we deduce that $C$ contains exactly two non-adjacent vertices of degree 3. Without loss of generality, suppose that $d(v_1)=d(v_3)=3$ and $d(v_2)=d(v_4)=4$. Since $g(G)>3$ by \Cref{l3}, $v_1$ and $v_3$ are not adjacent. Let $x_1 \in N(v_1)\backslash \{v_2,v_4\}$ and $x_3 \in N(v_3)\backslash \{v_2,v_4\}$. By \Cref{l5}, we deduce that $d(x_1)=d(x_3)=4$. We will show that $d(x_1,v_3)=3$. \\
If $d(x_1, v_3)=1$, then we get a vertex $v_3$ contained in two $4$-cycles, a contradiction by \Cref{l5}. So, $x_1\neq x_3$.  \\
If $d(x_1,v_3)=2$, then $x_1$ and $v_3$ share a common neighbor. We have $N(v_3)=\{v_2,v_4,x_3\}$. Since $g(G)>3$, $v_2 \notin N(x_1)$ and $v_4 \notin N(x_1)$. So, $x_3\in N(x_1)$. By \Cref{l4}, we deduce that $d(x_1)=d(x_3)=4$. Since $G$ is $4$-irregular and $x_1$ and $x_3$ are adjacent $4$-vertices, we obtain a contradiction.
\\
Thus, we have $d(x_1, v_3)=3$. Similarly, we have  $d(v_1, x_3)=3$.\\
By minimality of $|E(G)|$, we can greedily color all the vertices in $G^2$ except $V(C)$, $x_1$, and $x_3$. We have $|L(v_1) |\ge 4$, $|L(v_2)| \ge 5$, $|L(v_3) |\ge 4$, $|L(v_4)|\ge 5$, $|L(x_1) |\ge 2$, and $|L(x_3)| \ge 2$. \\
Since $d(x_1, v_3)=3$ and $|L(x_1)|
+|L(v_3)|> 4$, we can color $x_1$ and $v_3$ by $c_1$ and $c_2$, respectively, such that $|L(v_1)\backslash \{c_1,c_2\}| \ge 3$. Color $x_1$ and $v_3$ by $c_1$ and $c_2$, respectively, and call the new list $L_1$. We have $|L_1(v_1) |\ge 3$, $|L_1(v_2)| \ge 3$, $|L_1(v_4)|\ge 3$, and $|L_1(u_3)| \ge 1$. Recall that $d(v_1, x_3)=3$ and so coloring $x_3$ in $G^2$ does not affect $|L(v_1)|$.
Greedily color in order $x_3$, $v_2$, $v_4$, and $v_1$ in $G^2$ to obtain $\chi_l(G^2) \leq 11$, a contradiction. \end{proof}
\begin{lemma} \label{l7} $G$ has no two $5$-cycles sharing two consecutive edges. \end{lemma}
\begin{proof} Suppose, to the contrary, that $G$ has two $5$-cycles sharing two consecutive edges. Let $C=v_1v_2v_3v_4v_7$ and $C'= v_1v_7v_4v_5v_6$ be these two $5$-cycles sharing two consecutive edges $v_1v_7$ and $v_7v_4$ (See Figure \ref{fig:double_graph}). 
\begin{itemize}
\item \textbf{Case 1:} $d(v_1)=3$ and $d(v_4)=4$. \\
Then, since $G$ is $4$-irregular, we have $d(v_3)= d(v_5)=d(v_7)=3$. Since $v_1$ has at most one $3$-neighbor by \Cref{l2} and $d(v_7)=3$, we have $d(v_6)=d(v_2)=4.$ 
\begin{itemize}
\item \textbf{Subcase 1:} $d(v_3,v_6)=3$. \\
By minimality of $|E(G)|$, we can color all the vertices of $G^2$ except $V(C) \cup V(C')$. Now, we have $|L(v_1)|\ge 6$, $|L(v_2)|\ge 4$, $|L(v_3)|\ge 4$, $|L(v_4)|\ge 5$, $|L(v_5)|\ge 4$, $|L(v_6)|\ge 4$, and $|L(v_7)|\ge 6 $. Since $d(v_3,v_6)=3$ and $|L(v_3)| +|L(v_6)|> 6$, we can color $v_3$ by $c_1$ and $v_6$ by $c_2$ such that $|L(v_1)\backslash \{c_1, c_2\}| \ge 5$. Color
$v_3$ by $c_1$ and $v_6$ by $c_2$ and call the new list $L_1$. We have $|L_1(v_1)|\ge 5$, $|L_1(v_2)|\ge 2$, $|L_1(v_4)|\ge 3$, $|L_1(v_5)|\ge 2$, and $|L_1(v_7)|\ge 4 $. Now, we greedily color in order $v_2$, $v_5$, $v_4$, $v_7$, and $v_1$ in $G^2$ to obtain $\chi_l(G^2) \leq 11$, a contradiction.
\item \textbf{Subcase 2:} $d(v_2,v_5)=3$. \\
By minimality of $|E(G)|$, we can color all the vertices of $G^2$ except $V(C) \cup V(C')$. Now, we have $|L(v_1)|\ge 6$, $|L(v_2)|\ge 4$, $|L(v_3)|\ge 4$, $|L(v_4)|\ge 5$, $|L(v_5)|\ge 4$, $|L(v_6)|\ge 4$, and $|L(v_7)|\ge 6 $. Since $d(v_2,v_5)=3$ and $|L(v_2)| +|L(v_5)|> 6$, we can color $v_2$ by $c_1$ and $v_5$ by $c_2$ such that $|L(v_1)\backslash \{c_1, c_2\}| \ge 5$. Color
$v_2$ by $c_1$ and $v_5$ by $c_2$ and call the new list $L_1$. We have $|L_1(v_1)|\ge 5$, $|L_1(v_3)|\ge 2$, $|L_1(v_4)|\ge 3$, $|L_1(v_6)|\ge 2$, and $|L_1(v_7)|\ge 4 $. Now, we greedily color in order $v_3$, $v_6$, $v_4$, $v_7$, and $v_1$ in $G^2$ to obtain $\chi_l(G^2) \leq 11$, a contradiction. 
\item \textbf{Subcase 3:} $d(v_3,v_6) \leq 2$ and $d(v_2,v_5) \leq 2$. \\
Since $g(G) >4 $ by \Cref{l6}, we have $d(v_3,v_6) = 2$ and 
$d(v_2,v_5) = 2$. Let $x_1 \in N(v_3) \cap N(v_6)$ and $x_2 \in N(v_2) \cap N(v_5)$. Since $G$ is $4$-irregular, we have $d(x_1)=d(x_2)=3$. Let $x_6 \in N(v_6) \backslash \{v_1, v_5,x_1\}$. Note that $x_6 \neq v_2$ and $x_6 \neq v_4$ since $g(G)>3$. Now, we color all the vertices of $G^2$ except $V(C) \cup V(C') \cup \{x_1, x_2,x_6\}$. We have $|L(v_1)|\ge 9$, $|L(v_2)|\ge 7$, $|L(v_3)|\ge 8$, $|L(v_4)|\ge 7$, $|L(v_5)|\ge 9$, $|L(v_6)|\ge 8$, $|L(v_7)|\ge 6 $, $|L(x_1)| \ge 7 $, $|L(x_2)| \ge 6$, and $|L(x_6)| \ge 3$. \\
Since $N(v_3)= \{x_1, v_2, v_4\}$, we deduce that $d(x_6, v_3) >1$. Moreover, if $x_6$ and $v_3$ share a common neighbor, we get a $3$-cycle $x_1v_6x_6$ or a $4$-cycle  $v_4v_5v_6x_6$ or $v_6x_6v_2v_1$, a contradiction since $g(G)>4$. Thus, we have $d(x_6,v_3) =3$. Since $d(v_3,x_6)=3$ and $|L(v_3)| +|L(v_6)|> 9$, we can color $v_3$ by $c_1$ and $v_6$ by $c_2$ such that $|L(v_1)\backslash \{c_1, c_2\}| \ge 8$. Color
$v_3$ by $c_1$ and $x_6$ by $c_2$ and call the new list $L_1$. Now, we have $|L_1(v_1)|\ge 8$, $|L_1(v_2)|\ge 6$, $|L_1(v_4)|\ge 6$, $|L_1(v_5)|\ge 7$, $|L_1(v_6)|\ge 6$, $|L_1(v_7)|\ge 5 $, $|L_1(x_1)| \ge 5 $, and $|L_1(x_2)| \ge 5$. Now, we greedily color in order $x_1$, $x_2$, $v_2$, $v_4$, $v_5$, $v_6$, $v_7$, and $v_1$ in $G^2$ to obtain $\chi_l(G^2) \leq 11$, a contradiction.
\end{itemize}
\item \textbf{Case 2:} $d(v_1)=3$ and $d(v_4)=3$. Since $G$ is 4-irregular and a $3$-vertex has at most one $3$-neighbor by \Cref{l2}, we deduce that $d(v_7)=4$. Moreover, exactly one of $v_2$ and $v_3$ is of degree $3$. Without loss of generality, suppose that $d(v_2)=3$ and $d(v_3)=4$. Thus, we deduce in this case that $d(v_6)=4$ and $d(v_5)=3$. 
\begin{itemize}
\item \textbf{Subcase 1:} $d(v_3,v_6)=3$. \\
By minimality of $|E(G)|$, we can color all the vertices of $G^2$ except $V(C) \cup V(C')$. Now, we have $|L(v_1)|\ge 6$, $|L(v_2)|\ge 5$, $|L(v_3)|\ge 4$, $|L(v_4)|\ge 6$, $|L(v_5)|\ge 5$, $|L(v_6)|\ge 4$, and $|L(v_7)|\ge 5 $. Since $d(v_3,v_6)=3$ and $|L(v_3)| +|L(v_6)|> 6$, we can color $v_3$ by $c_1$ and $v_6$ by $c_2$ such that $|L(v_1)\backslash \{c_1, c_2\}| \ge 5$. Color
$v_3$ by $c_1$ and $v_6$ by $c_2$ and call the new list $L_1$. We have $|L_1(v_1)|\ge 5$, $|L_1(v_2)|\ge 3$, $|L_1(v_4)|\ge 4$, $|L_1(v_5)|\ge 3$, and $|L_1(v_7)|\ge 3$. Now, we greedily color in order $v_7$, $v_5$, $v_2$, $v_4$, and $v_1$ in $G^2$ to obtain $\chi_l(G^2) \leq 11$, a contradiction. 
\item \textbf{Subcase 2:} $d(v_2,v_5)=3$. \\
By minimality of $|E(G)|$, we can color all the vertices of $G^2$ except $V(C) \cup V(C')$. Now, we have $|L(v_1)|\ge 6$, $|L(v_2)|\ge 5$, $|L(v_3)|\ge 4$, $|L(v_4)|\ge 6$, $|L(v_5)|\ge 5$, $|L(v_6)|\ge 4$, and $|L(v_7)|\ge 5 $. Since $d(v_2,v_5)=3$ and $|L(v_2)| +|L(v_5)|> 6$, we can color $v_2$ by $c_1$ and $v_5$ by $c_2$ such that $|L(v_1)\backslash \{c_1, c_2\}| \ge 5$. Color
$v_2$ by $c_1$ and $v_5$ by $c_2$ and call the new list $L_1$. We have $|L_1(v_1)|\ge 5$, $|L_1(v_3)|\ge 2$, $|L_1(v_4)|\ge 4$, $|L_1(v_6)|\ge 2$, and $|L_1(v_7)|\ge 3$. Now, we greedily color in order $v_3$, $v_6$, $v_7$, $v_4$, and $v_1$ in $G^2$ to obtain $\chi_l(G^2) \leq 11$, a contradiction.

\item \textbf{Subcase 3:} $d(v_3,v_6) \leq 2$ and $d(v_2,v_5) \leq 2$. \\
Since $g(G) >4 $ by \Cref{l6}, we have $d(v_3,v_6) = 2$ and 
$d(v_2,v_5) = 2$. Let $x_1 \in N(v_3) \cap N(v_6)$ and $x_2 \in N(v_2) \cap N(v_5)$. Since $G$ is $4$-irregular, we have $d(x_1)=d(x_2)=3$. Let $x_6 \in N(v_6) \backslash \{v_1, v_5,x_1\}$. Note that $x_6 \neq v_2$ and $x_6 \neq v_4$ since $g(G)>3$. Now, we color all the vertices of $G^2$ except $V(C) \cup V(C') \cup \{x_1, x_2,x_6\}$. We have $|L(v_1)|\ge 9$, $|L(v_2)|\ge 8$, $|L(v_3)|\ge 7$, $|L(v_4)|\ge 8$, $|L(v_5)|\ge 9$, $|L(v_6)|\ge 8$, $|L(v_7)|\ge 5$, $|L(x_1)| \ge 6 $, $|L(x_2)| \ge 5$, and $|L(x_6)| \ge 3$. \\
Since $N(v_3)= \{x_1, v_2, v_4\}$, we deduce that $d(x_6, v_3) >1$. Moreover, if $x_6$ and $v_3$ share a common neighbor, we get a $3$-cycle $x_1v_6x_6$ or a $4$-cycle  $v_4v_5v_6x_6$ or $v_6x_6v_2v_1$, a contradiction since $g(G)>4$. Thus, we have $d(x_6,v_3) =3$. Since $d(v_3,x_6)=3$ and $|L(v_3)| +|L(v_6)|> 9$, we can color $v_3$ by $c_1$ and $v_6$ by $c_2$ such that $|L(v_1)\backslash \{c_1, c_2\}| \ge 8$. Color
$v_3$ by $c_1$ and $x_6$ by $c_2$ and call the new list $L_1$. Now, we have $|L_1(v_1)|\ge 8$, $|L_1(v_2)|\ge 7$, $|L_1(v_4)|\ge 7$, $|L_1(v_5)|\ge 7$, $|L_1(v_6)|\ge 6$, $|L_1(v_7)|\ge 4 $, $|L_1(x_1)| \ge 4 $, and $|L_1(x_2)| \ge 4$. Now, we greedily color in order $x_1$, $x_2$, $v_7$, $v_6$, $v_5$, $v_4$, $v_2$, and $v_1$ in $G^2$ to obtain $\chi_l(G^2) \leq 11$, a contradiction.
\end{itemize}

\item \textbf{Case 3:} $d(v_1)=4$ and $d(v_4)=4$. \\
Since $G$ is $4$-irregular, we deduce that $d(v_2)=d(v_3)=d(v_5)=d(v_6)=d(v_7)=3$. \\
If $d(v_6,v_3)=1$, we get a $4$-cycle. So, $d(v_6,v_3)\ge 2$ and similarly $d(v_2,v_5)\ge 2$. We will show that $d(v_2,v_5)=d(v_3,v_6)=3$. 

Suppose first that $d(v_3,v_6)=2$. We distinguish two cases: (1) $d(v_2,v_5)=3$, and (2) $d(v_2,v_5)=2$.

Suppose that  $d(v_2,v_5)=3$. Let $x_1 \in N(v_3) \cap N(v_6)$. Now, we color all the vertices of $G^2$ except $V(C) \cup V(C') \cup \{x_1\}$. We have $|L(v_1)|\ge 6$, $|L(v_2)|\ge 6$, $|L(v_3)|\ge 7$, $|L(v_4)|\ge 6$, $|L(v_5)|\ge 6$, $|L(v_6)|\ge 7$, $|L(v_7)|\ge 5 $, and $|L(x_1)| \ge 5 $. 
Since $d(v_2,v_5)=3$ and $|L(v_2)| +|L(v_5)|> 7$, we can color $v_2$ by $c_1$ and $v_5$ by $c_2$ such that $|L(v_6)\backslash \{c_1, c_2\}| \ge 6$. Color
$v_2$ by $c_1$ and $v_5$ by $c_2$ and call the new list $L_1$. Now, we have $|L_1(v_1)|\ge 4$, $|L_1(v_3)|\ge 5$, $|L_1(v_4)|\ge 4$, $|L_1(v_6)|\ge 6$, $|L_1(v_7)|\ge 3 $, and $|L_1(x_1)| \ge 3$. Now, we greedily color in order $x_1$, $v_7$, $v_4$, $v_1$, $v_3$, and $v_6$ in $G^2$ to obtain $\chi_l(G^2) \leq 11$, a contradiction.

Suppose now that $d(v_2,v_5)=2$.
Let $x_1 \in N(v_3) \cap N(v_6)$ and $x_2 \in N(v_2) \cap N(v_5)$. Since $G$ is $4$-irregular, we have $d(x_1)=d(x_2)=3$. Let $y_1 \in N(v_1) \backslash \{v_1, v_6,v_7\}$. Note that $y_1 \notin \{x_1,x_2,v_3,v_4\}$ since $g(G)>4$. Now, we color all the vertices of $G^2$ except $V(C) \cup V(C') \cup \{x_1, x_2,y_1\}$. We have $|L(v_1)|\ge 8$, $|L(v_2)|\ge 9$, $|L(v_3)|\ge 8$, $|L(v_4)|\ge 7$, $|L(v_5)|\ge 8$, $|L(v_6)|\ge 9$, $|L(v_7)|\ge 6 $, $|L(x_1)| \ge 5 $, $|L(x_2)| \ge 5$, and $|L(y_1)| \ge 3$. \\
Since $N(v_5)= \{x_2, v_4, v_6\}$, we deduce that $d(y_1, v_5) >1$. Moreover, if $y_1$ and $v_5$ share a common neighbor, we get a $3$-cycle $y_1v_6x_6$ or a $4$-cycle  $y_1x_2v_2v_1$ or $y_1v_1v_7v_4$, a contradiction since $g(G)>4$. Thus, we have $d(y_1,v_5) =3$. Since $d(y_1,v_5)=3$ and $|L(y_1)| +|L(v_5)|> 9$, we can color $v_3$ by $c_1$ and $v_6$ by $c_2$ such that $|L(v_6)\backslash \{c_1, c_2\}| \ge 8$. Color
$v_5$ by $c_1$ and $y_1$ by $c_2$ and call the new list $L_1$. Now, we have $|L_1(v_1)|\ge 6$, $|L_1(v_2)|\ge 7$, $|L_1(v_3)|\ge 7$, $|L_1(v_4)|\ge 6$, $|L_1(v_6)|\ge 8$, $|L_1(v_7)|\ge 4 $, $|L_1(x_1)| \ge 4$, and $|L_1(x_2)| \ge 4$. Now, we greedily color in order $x_1$, $x_2$, $v_1$, $v_7$, $v_4$, $v_3$, $v_2$, and $v_6$ in $G^2$ to obtain $\chi_l(G^2) \leq 11$, a contradiction.

Thus, $d(v_6,v_3)=3$ and similarly $d(v_2,v_5)=3$. Color all the vertices of $G$ in $G^2$ except $V(C) \cup V(C')$. We have $|L(v_i)| \ge 5$ for all $ i \in \{1,2,3,4,5,6,7\}$.
\begin{itemize}
\item \textbf{Subcase 1:}  Either $L(v_2) \cap L(v_5) \neq \emptyset $ or $L(v_3) \cap L(v_6) \neq \emptyset$. \\
Without loss of generality, suppose that 
$L(v_2) \cap L(v_5) \neq \emptyset $. Color $v_2$ and $v_5$ by $c\in L(v_2) \cap L(v_5)$ and call the new list $L_1$. We have $|L_1(v_i)| \ge 4$ for all $ i \in \{1,3,4,6,7\}$. We greedily in order $v_1$, $v_6$, $v_7$, $v_4$, and $v_3$  in $G^2$ to obtain $\chi_l(G^2) \leq 11$, a contradiction.
\item \textbf{Subcase 2:} $L(v_2) \cap L(v_5)=L(v_3) \cap L(v_6) = \emptyset$.
Thus, $|L(v_2) \cup L(v_5)| \ge 10$. Then, we can color $v_2$ by $c_2$ and $v_5$ by $c_5$ such that $|L(v_3)\backslash \{c_2,c_5\}| \ge 4$ and $|L(v_6)\backslash \{c_2,c_5\}| \ge 4$. Color $v_2$ by $c_2$ and $v_5$ by $c_5$ and call the new list $L_1$. We have $|L_1(v_1)| \ge 3$, $|L_1(v_4)| \ge 3$, $|L_1(v_7)| \ge 3$, $|L_1(v_3)| \ge 4$, and $|L_1(v_6)| \ge 4$. Then, we greedily color in order $v_1$, $v_4$, $v_7$, $v_3$, and $v_6$ in $G^2$ to obtain $\chi_l(G^2) \leq 11$, a contradiction.
\end{itemize}
\end{itemize}
\end{proof}
\begin{figure}[h!]
\centering
\begin{tikzpicture}[scale=1.7, every node/.style={circle, fill=black, inner sep=1.8pt}]
  \node (v1) at (1.6,0) [label=right:$v_1$] {};
  \node (v2) at (1.6,1.5) [label=right:$v_2$] {};
  \node (v3) at (0,1.5) [label=left:$v_3$] {};
  \node (v4) at (0,0) [label=left:$v_4$] {};
  \node (v5) at (0.45,0.8) [label=above left:$v_5$] {};
  \node (v6) at (1.15,0.8) [label=above right:$v_6$] {};
  \node (v7) at (0.8,-0.9) [label=below:$v_7$] {};
  \draw (v1)--(v2)--(v3)--(v4)--(v7)--(v1);
  \draw (v4)--(v5)--(v6)--(v1);
\end{tikzpicture}
\hspace{1.5cm} 
\begin{tikzpicture}[scale=1.8, every node/.style={circle, fill=black, inner sep=1.8pt}]
  \node (v1) at (1.6,0) [label=right:$v_1$] {};
  \node (v2) at (1.6,1.5) [label=right:$v_2$] {};
  \node (v3) at (0,1.5) [label=left:$v_3$] {};
  \node (v4) at (0,0) [label=left:$v_4$] {};
  \node (v5) at (0.45,0.8) [label=above left:$v_5$] {};
  \node (v6) at (1.15,0.8) [label=above right:$v_6$] {};
  \node (v7) at (0.8,-0.9) [label=below:$v_7$] {};

  \node (x1) at (0.55,1.1) [label=above:$x_1$] {};
  \node (x2) at (1.0,1.2) [label=above:$x_2$] {};

  \node (x6) at (1.15,0.25) [label=right:$x_6$] {};

  \draw (v1)--(v2)--(v3)--(v4)--(v7)--(v1);
  \draw (v4)--(v5)--(v6)--(v1);
  \draw (x1)--(v3);
  \draw (x1)--(v6);
  \draw (x2)--(v2);
  \draw (x2)--(v5);
  \draw (x6)--(v6);
\end{tikzpicture}

\caption{Illustration of \Cref{l7}}
\label{fig:double_graph}
\end{figure} 
\begin{lemma} \label{l8} $g(G) \ge 6$. \end{lemma}
\begin{proof} Suppose that $G$ has a $5$-cycle $C=v_1v_2v_3v_4v_5$. Since $G$ is $4$-irregular and $C$ is odd, we deduce that there exists $i \in \{1,2,3,4\}$ such that $d(v_i)=d(v_{i+1})=3$ or $d(v_1)=d(v_5)=3$.  Without loss of generality, suppose that $d(v_1)=d(v_2)=3$. Then, we deduce that $d(v_3)=d(v_5)=4$ by \Cref{l2}. Since $G$ is $4$-irregular, we have $d(v_4)=3$.  Let $x_1 \in N(v_1)\backslash \{v_2,v_5\}$ and  $ x_3 \in N(v_3)\backslash \{v_2,v_4\}$. Now, we color all the vertices in $G^2$ except $V(C) \cup \{x_1,x_3\}$. We have $|L(v_1)| \ge 5$, $|L(v_2)| \ge 6$, $|L(v_3)| \ge 4$, $|L(v_4)| \ge 4$, $|L(v_5)| \ge 4$, $|L(x_1)| \ge 2$ and $|L(x_3)| \ge 2$. If $x_1=x_3$, we get a $4$-cycle in $G$ which contradicts \Cref{l6}. If $d(x_1,v_3)=2$, we get two $5$-cycles $C$ and $v_1v_2v_3x_3x_1$ sharing two consecutive edges which contradicts \Cref{l7}. Thus, we have $d(x_1,v_3)=3$. Since $|L(x_1)| + |L(v_3)| >5$, we can color $x_1$ by $c_1$ and $v_3$ by $c_3$ such that $|L(v_1)\backslash \{c_1,c_3\}|\ge 4$. Color $x_1$ by $c_1$ and $v_3$ by $c_3$ and call the new list $L_1$. We have $|L_1(v_1)| \ge 4$, $|L_1(v_2)| \ge 4$, $|L_1(v_4)| \ge 3$, $|L_1(v_5)| \ge 2$ and $|L_1(x_3)| \ge 1$. So, we  greedily color in order $x_3$, $v_5$, $v_4$, $v_2$, and $v_1$ in $G^2$ to obtain $\chi_l(G^2) \leq 11$, a contradiction. \end{proof}
\begin{lemma} \label{l9} $g(G) \ge 7$. \end{lemma}
\begin{proof} Suppose that $G$ has a $6$-cycle $C=v_1v_2v_3v_4v_5v_6$. By minimality of $|E(G)|$, we can color all the vertices of $G$ but $V(C)$. Then, $|L(v_i)| \ge 3$ $\forall i, 1 \leq i \leq 6$. Since $g(G) >5$ by \Cref{l8}, we have $d(v_i,v_{i+3})=3$ for $i=1,2,3$. 
\begin{itemize}
\item \textbf{Case 1:} There exists $i \in \{1,2,3\}$ such that $L(v_i) \cap L(v_{i+3}) \neq \emptyset $. \\
Without loss of generality, suppose that $L(v_1) \cap L(v_{4}) \neq \emptyset $. Color $v_1$ and $v_4$ by $c \in L(v_1)\cap L(v_4)$
and call the new list $L_1$. We have $|L_1(v_i)|\ge 2$ for $i= 2,3,5,6$. 

Suppose that $L_1(v_3) \cap L_1(v_6) \neq \emptyset $. Then, we color $v_3$ and $v_6$ by $c' \in L_1(v_3)\cap L_1(v_6) $ in $G^2$ and call the new list $L_2$. Now, we have $|L_2(v_2)| \ge 1$ and $|L_2(v_5)| \ge 1$. Since $d(v_2,v_5)=3$, we can greedily color $v_2$ and $v_5$ independently by an available color in $G^2$ to obtain $\chi_l(G^2) \leq 11$, a contradiction. 

Suppose now that $L_1(v_3) \cap L_1(v_6) = \emptyset $. Now, we color $v_2$ by $c_1\in L_1(v_2)$ in $G^2$. Since $L_1(v_3) \cap L_1(v_6) = \emptyset $, either $|L_1(v_3)\backslash \{c_1\}| \ge 2$ or $|L_1(v_6)\backslash \{c_1\}| \ge 2$. If $|L_1(v_3)\backslash \{c_1\}| \ge 2$, we greedily color in order $v_6$, $v_5$ and $v_3$ in $G^2$ to obtain $\chi_l(G^2) \leq 11$, a contradiction. If 
$|L_1(v_6)\backslash \{c_1\}| \ge 2$, we greedily color in order $v_3$, $v_5$ and $v_6$ in $G^2$ to obtain $\chi_l(G^2) \leq 11$, a contradiction.
\item \textbf{Case 2:} We have  $L(v_1) \cap L(v_4) = L(v_2) \cap L(v_5) = L(v_3) \cap L(v_6) = \emptyset $. \\
If $L(v_i)=L(v_{i+1})$ for all $ i \in \{1,2,3,4,5\}$, then we have $L(v_i) \cap L(v_{i+3}) \neq \emptyset $ for $i=1,2,3$, a contradiction. Thus, there exists $i \in \{1,2,3,4,5\}$ such that $L(v_i) \neq L(v_{i+1})$. Without loss of generality, suppose that $L(v_1) \neq L(v_2)$. Then, either there exists $c \in L(v_1)\backslash L(v_2)$ or $ c \in L(v_2)\backslash L(v_1)$. Without loss of generality, suppose that  there exists $ c \in L(v_1)\backslash L(v_2)$. Color $v_1$ by $c$ and call the new list $L_1$. We have $|L_1(v_2)|\ge 3$, $|L(v_3)|\ge 2$, $|L_1(v_4)|\ge 3$ and  $|L_1(v_5)| \ge 2$, $|L_1(v_6)|\ge 2$. Since $L(v_3) \cap L(v_6) =  \emptyset$, either $c \notin L(v_3)$ or $c \notin L(v_6)$. 
\begin{itemize}
\item \textbf{Subcase 1:} $c \notin L(v_3)$. \\ 
Then, $|L_1(v_3)|\ge 3$ and  $|L_1(v_6)| \ge 2$.

We will show that $L_1(v_6)\subset L_1(v_2)$ and $L_1(v_4)\subset L_1(v_2)$. 

Suppose that $L_1(v_6)\not\subset L_1(v_2)$. Then, there exists $c_6\in L_1(v_6)\backslash L_1(v_2)$.  Color $v_6$ by $c_6$ and call the new list $L_2$. So, we have $|L_2(v_2)|\ge 3$, $|L_2(v_3)|\ge 3$,  $|L_2(v_4)| \ge 2$, and $|L_2(v_5)| \ge 1$. Greedily color in order $v_5$, $v_4$, $v_3$,
and $v_2$ in $G^2$ to obtain $\chi_l(G^2) \leq 11$, a contradiction. So,  $L_1(v_6) \subset L_1(v_2)$. 

Suppose that  $L_1(v_4)\not \subset L_1(v_2)$. Then, there exists $c_4\in L_1(v_4)\backslash L_1(v_2)$. Color $v_4$ by $c_4$ and call the new list $L_2$. Note that $c_4 \notin L_1(v_6)$ since $L_1(v_6) \subset L_1(v_2)$. So, we have $|L_2(v_2)|\ge 3$, $|L_2(v_3)|\ge 2$,  $|L_2(v_5)| \ge 1$, and $|L_2(v_6)| \ge 2$. Greedily color in order $v_5$, $v_6$, $v_3$,
and $v_2$ in $G^2$ to obtain $\chi_l(G^2) \leq 11$, a contradiction. So,  $L_1(v_4) \subset L_1(v_2)$. 

Since $|L_1(v_3)|+|L_1(v_6)|>3$, $L_1(v_3)\cap L_1(v_6)=\emptyset$, and $L_1(v_6) \subset L_1(v_1)$, we deduce that there exists $c_3\in L_1(v_3)$ such that $|L_1(v_2)\backslash \{c_3\}|\ge 3$. Since $L_1(v_4) \subset L_1(v_2)$, we have $|L_1(v_4)\backslash \{c_3\}|\ge 3$. Color $v_3$ by $c_3$ then greedily color in order $v_5$, $v_6$, $v_4$,
and $v_2$ in $G^2$ to obtain $\chi_l(G^2) \leq 11$, a contradiction.
\item \textbf{Subcase 2:} $c \notin L(v_6)$. \\
Then, $|L_1(v_6)|\ge 3$ and  $|L_1(v_3)| \ge 2$. 

We will show that $ L_1(v_6)\neq L_1(v_2)$.

Suppose that $ L_1(v_6)= L_1(v_2)$. Then, since $L_1(v_3) \cap L_1(v_6)= \emptyset$, we deduce that $L_1(v_2) \cap L_1(v_3)= \emptyset$. Color $v_3$ by $c_3$
and call the new list $L_2$. So, $c_3 \notin L_1(v_2)$. Thus, we have $|L_2(v_2)|\ge 3$, $|L_2(v_4)|\ge 2$,  $|L_2(v_5)| \ge 1$ and $|L_2(v_6)| \ge 3$. Thus, we can greedily color in order $v_5$, $v_4$, $v_6$, and $v_2$ in $G^2$ to obtain $\chi_l(G^2) \leq 11$, a contradiction. Hence, $ L_1(v_6) \neq L_1(v_2)$. 

Greedily color in order $v_5$, $v_3$ and $v_4$ in $G^2$ and call the new list $L_2$. Since $ L_1(v_6) \neq L_1(v_2)$, either $|L_2(v_2)|\ge 2$ and $|L_2(v_6)|\ge 1$, or $|L_2(v_6)|\ge 2$ and $|L_2(v_1)|\ge 1$, or $|L_2(v_2)|=|L_2(v_6)|= 1$ but $L_2(v_1) \neq L_2(v_6)$. Thus, in all cases we can extend the coloring to the whole graph of $G^2$ to obtain $\chi_l(G^2) \leq 11$, a contradiction.
\end{itemize}
\end{itemize}
\end{proof}
\begin{lemma} \label{l10} Let $C$ be the shortest cycle in $G$. Then, for any $v_i,v_j \in V(C)$, we have $d(v_i,v_j) \ge 3 $ when $j \notin \{i-2,i-1,i,i+1,i+2\}$. \end{lemma}
\begin{proof} Let $C$ be the shortest cycle in $G$. Set $C= v_1v_2...v_k$. By \Cref{l9}, $k \ge 7$. Suppose that there exists two vertices $v_i$ and $v_j$ in $C$ such that $j \notin \{i-2,i-1,i,i+1,i+2\}$ and $d(v_i,v_j) \leq 2$. Suppose that $d(v_i,v_j)=1$. Then, either  $v_iv_{i+1}\dots v_j$ or $v_iv_{i-1}\dots v_j$ is a cycle $C'$ with $l(C')<l(C)$, a contradiction. Suppose now that $d(v_i,v_j)=2$. Let $x\in N(v_i)\cap N(v_j)$. Then, either  $v_iv_{i+1}\dots v_jx$ or $v_iv_{i-1}\dots v_jx$ is a cycle $C'$ with $l(C')<l(C)$, a contradiction.
\end{proof}

Now, we are ready to prove that $\chi_l(G^2) \leq 11$, a contradiction. \\

\subsection{Eliminating the Minimal Counterexample}
Now, we will study a shortest cycle in $G$ and analyze the interaction between the available  color lists on the cycle and its external neighbors. This structural information allows a sequence of recoloring arguments leading to the desired contradiction. 
\\

Let $C = v_1v_2v_3...v_k$ be the shortest cycle in $G$. Then, $C$ has no chords and $k \ge 7$ since $g(G) \ge 7$ by \Cref{l9}. Moreover, since $C$ is the shortest cycle, then we have $d(v_i,v_j) \ge 3 $ for $j \notin \{i-2,i-1,i+1,i+2\}$ by \Cref{l10}. Let $x_i \in N(v_i) \backslash V(C)$. 
Color all the vertices in $G^2$ except $V(C) \cup \{x_i\; ;i \in \mathbb{N}\}$.
We have $|L(v_i)| \ge 6$ and $|L(x_i)| \ge 2$ for all $i\in \{1,\dots ,k\}$. Note that $d(x_i,v_{i+2})=3$, $d(x_i,x_{i+1})=3$, and $d(x_i,x_{i+2})\ge 3$ since $g(G) \ge 7$.
\begin{claim} \label{c1} $L(x_i) \subset L(v_i)$ for all $i$, $1\leq i \leq k$. 
\end{claim}
\begin{proof} Suppose, to the contrary, that there exists $c \in L(x_i) \backslash L(v_i)$ for some $i$. Color $x_i$ by $c$ and call the new list $L_1$. So, $|L_1(v_i)|\ge 6$, $|L_1(v_{i+1})|\ge 5$, and $|L_1(v_{i-1})|\ge 5$. 

Note that $d(v_{i+1},x_{i-1})\ge 3$ since $g(G)>6$.
We will show that $L_1(v_{i+1}) \cap L_1(x_{i-1})=\emptyset$. 

Suppose, to the contrary, that $L_1(v_{i+1}) \cap L_1(x_{i-1}) \neq \emptyset$. Then, we can color $v_{i+1}$ and $x_{i-1}$ by $c' \in L_1(v_{i+1}) \cap L_1(x_{i-1}) $. Since each of the uncolored vertices in $V(C)\cup \{x_i,i \in \mathbb{N}\}$ has at least two uncolored neighbors in $G^2$, we can greedily color in order $x_{i+1}$, $x_{i+2}$,..., $x_{i-2}$, $v_{i+2}$, $v_{i+3}$, ..., and $v_{i-2}$ in $G^2$. Call the new list $L_2$. We have $|L_2(v_i)| \ge 2$ and $|L_2(v_{i-1})| \ge 1$. Greedily color $v_{i-1}$ then $v_i$ in $G^2$ to obtain $\chi_l(G^2) \leq 11$, a contradiction. 

Thus, we deduce that $L_1(v_{i+1}) \cap L_1(x_{i-1})=\emptyset$. Since $|L_1(v_{i+1}) \cup L_1(x_{i-1})| >5$, we can either color $v_{i+1}$ by some color $c_1$ or $x_{i-1}$ by some color $c_2$
such that $|L_1(v_{i-1})\backslash \{c_1\}| \ge 5$ or 
$|L_1(v_{i-1})\backslash \{c_2\}| \ge 5$. 
\begin{itemize}
\item \textbf{Case 1:} $|L_1(v_{i-1})\backslash \{c_1\}| \ge 5$. \\
Color $v_{i+1}$ by $c_1$ and call the new list $L_2$.

\begin{itemize}
\item \textbf{Subcase 1:} Suppose that for any choice of $c_1$, we have $|L_1(v_1)\backslash \{c_1\}|=5$. 

We will show that either $L_2(v_i)\neq L_2(v_{i-1})$ or we can color $x_{i-1}$ by some color $c_3$ such that $|L_2(v_i)\backslash \{c_3\}|= 5$. 

Suppose that $L_2(v_i)= L_2(v_{i-1})$. Then, $c_1$ is unique. So, for any $c_1'\in L_1(v_{i+1})\backslash \{c_1\}$, we have $c_1'\in L_1(v_i)$. Thus, since $L_1(v_{i+1})\cap L_1(x_{i-1})=\emptyset$ and $|L_1(v_{i+1})| +|L_1(x_{i-1})|>6$, we deduce that we can color $x_{i-1}$ by some color $c_3$ such that $|L_2(v_i)\backslash \{c_3\}|\ge 5$.

Assume first that $L_2(v_i)\neq L_2(v_{i-1})$. Now greedily color in order $x_{i+1}$, $x_{i+2}$, ..., $x_{i-1}$, $v_{i+2}$, $v_{i+3}$,... $v_{i-2}$ in $G^2$ which is possible since each vertex has at least two uncolored neighbors in $G^2$. Call this list $L_3$. Since $L_2(v_i)\neq L_2(v_{i-1})$, we deduce that either $|L_3(v_i)| \ge 2$ and $|L_3(v_{i-1})| \ge 1$, or  $|L_3(v_i)| \ge 1$ and $|L_3(v_{i-1})| \ge 2$, or $|L_3(v_i)|=|L_3(v_{i-1})|= 1$ but $L_3(v_i)\neq L_3(v_{i-1})$. Thus, we can greedily color $v_{i+1}$ and $v_i$ in both cases in $G^2$ to obtain $\chi_l(G^2) \leq 11$, a contradiction. 

Assume now that we can color $x_{i-1}$ by some color $c_3$ such that $|L_2(v_i)\backslash \{c_3\}|= 5$. Color $x_{i-1}$ by $c_3$ then greedily color in order $x_{i+1}$, $x_{i+2}$, ..., $x_{i-1}$, $v_{i+1}$, $v_{i+3}$,... $v_{i-2}$ in $G^2$ which is possible since each vertex has at least two uncolored neighbors in $G^2$. Call this list $L_3$. So, we have $|L_3(v_i)| \ge 2$ and $|L_3(v_{i-1})| \ge 1$. Thus, we can greedily color $v_{i+1}$ then $v_i$ in $G^2$ to obtain $\chi_l(G^2) \leq 11$, a contradiction.
\item \textbf{Subcase 2:} Suppose that we can choose $c_1$ such that $|L_1(v_1)\backslash \{c_1\}|\ge 6$.  Greedily color $x_{i+1}$, $x_{i+2}$, ..., $x_{i-1}$, $v_{i+2}$, $v_{i+3}$,... $v_{i-2}$ in order in $G^2$. Call this list $L_3$. We have $|L_3(v_i)| \ge 2$ and $|L_3(v_{i-1})| \ge 1$. Greedily color $v_{i+1}$ then $v_i$ in $G^2$ to obtain $\chi_l(G^2) \leq 11$, a contradiction. 
\end{itemize}
\item \textbf{Case 2:} $|L_1(v_{i-1})\backslash \{c_2\}| \ge 6$. \\
Color $x_{i-1}$ by $c_2$ and call the new list $L_2$.  
Note that $d(v_{i+1},x_{i-2})\ge 3$ since $g(G)>6$. Since $|L_2(v_{i+1})|+|L_2(x_{i-2})|>5$, we can color $v_{i+1}$ by $c_3$ and $x_{i-2}$ by $c_4$ such that $|L_2(v_{i-1})\backslash \{c_3,c_4\}|\ge 4$. Color  $v_{i+1}$ by $c_3$ and $x_{i-2}$ by $c_4$ then greedily color in order $x_{i+1}$, $x_{i+2}$,..., $x_{i-3}$, $v_{i+2}$, $v_{i+3}$,..., and $v_{i-2}$ and call the new list $L_3$. So, we have $|L_3(v_i)|\ge 1$ and $|L_3(v_{i-1})|\ge 2$. Greedily color $v_i$ then $v_{i-1}$ in $G^2$ to obtain $\chi_l(G^2)\leq 11$, a contradiction.
\end{itemize}
\end{proof}
\begin{claim} \label{c2} $L(x_i)\cap L(x_{i+1}) = \emptyset$ for all $i$. \end{claim}
\begin{proof} Assume that $L(x_i)\cap L(x_{i+1}) \neq \emptyset$. Recall that $d(x_i,x_{i+1})\ge 3$ since $g(G)\ge 7$. Color $x_i$ and $x_{i+1}$ by $c \in L(x_i)\cap L(x_{i+1})$ and call the new list $L_1$. So, $|L_1(v_i)|\ge 5$, $|L_1(v_{i+1})|\ge 5$, $|L_1(v_{i+2})|\ge 5$, and $|L_1(x_{i-1})|\ge 2$. Since $g(G) > 6$, we have $d(v_{i+2}, x_{i-1})\ge 3$. Since $|L_1(v_{i+2})| + |L_1(x_{i-1})| >5$ and $|L_1(v_i)| \ge 5$, we can color 
$v_{i+2}$ by $c_1$ and $x_{i-1}$ by $c_2$ such that $|L_1(v_i) \backslash \{c_1,c_2\}| \ge 4$. Color $v_{i+2}$ by $c_1$ and $x_{i-1}$ by $c_2$ then greedily color $x_{i+2}$, ..., $x_{i-2}$, $v_{i+3}$, $v_{i+4}$,..., and $v_{i-1}$ in order in $G^2$ which is possible since each vertex has at least two uncolored neighbors in $G^2$. Call this list $L_2$. We have $|L_2(v_i)| \ge 2$ and $|L_2(v_{i+1})| \ge 1$. Greedily color $v_{i+1}$ then $v_i$ in $G^2$ to obtain $\chi_l(G^2) \leq 11$, a contradiction.
\end{proof}
\begin{claim} \label{c3} $L(x_{i-1})\cap L(x_{i+1}) = \emptyset$ for all $i$.
\end{claim}
\begin{proof} Suppose, to the contrary, that $L(x_{i-1})\cap L(x_{i+1}) \neq \emptyset$. Note that $d(x_{i-1},x_{i+1})\ge 3$ since $g(G)>6$. Color $x_{i-1}$ and $x_{i+1}$ by $c \in L(x_{i-1})\cap L(x_{i+1})$ and call the new list $L_1$. Since $g(G) > 6$, we have $d(v_{i+2}, x_i)\ge 3$. 

We will show that $L_1(v_{i+2}) \cap L_1(x_i) = \emptyset$.

Suppose that $L_1(v_{i+2}) \cap L_1(x_i) \neq \emptyset $. Color $v_{i+2}$ and $x_i$ by $c \in L_1(v_{i+2}) \cap L_1(x_i)$. Now, greedily color in order $x_{i+2}$, $x_{i+3}$, ..., $x_{i-2}$, $v_{i+3}$, $v_{i+4}$,..., and $v_{i-1}$ in $G^2$ which is possible since each vertex has at least two uncolored neighbors in $G^2$. Call the new list $L_2$. We have $|L_2(v_i)| \ge 2$ and $|L_2(v_{i+1})| \ge 1$. Greedily color $v_{i+1}$ then $v_i$ in $G^2$ to obtain $\chi_l(G^2) \leq 11$, a contradiction. Then, we have $L_1(v_{i+2}) \cap L_1(x_i) = \emptyset$. 

Since $|L_1(v_{i+2})| + |L_1(x_i)| > 5$ and $d(v_{i+2},x_i)=3$, either there exists $ c_1 \in L_1(v_{i+2})$ such that $|L_1(v_{i}) \backslash \{c_1\}| \ge 5$ or there exists  $ c_2 \in L_1(x_{i})$  such that $|L_1(v_{i}) \backslash \{c_2\}| \ge 5$. \begin{itemize}
\item \textbf{Case 1:} There exists $ c_1 \in L_1(v_{i+2})$  such that $|L_1(v_{i}) \backslash \{c_1\}| \ge 5$. Then, color $v_{i+2}$
by $c_1$ and greedily color $x_{i+2}$ and $x_{i+3}$ and call the new list $L_2$. Since $g(G) >6$, we have $d(x_i, v_{i+3}) \ge 3$. We have $|L_2(x_i)|\ge 2$, $|L_2(v_{i+1})|\ge 3$, and $|L_2(v_{i+3})|\ge 3$. Since $|L_2(x_i)| + |L_2(v_{i+3})| > 4 $ and $|L_2(v_{i+1}) |\ge 4$, we can color $x_i$ by $c_3$ and $v_{i+3}$ by $c_4$ such that $|L_2(v_{i+1}) \backslash \{c_3,c_4\}|\ge 2$. Color  $x_i$ by $c_3$ and $v_{i+3}$ by $c_4$ then greedily color in order $x_{i+4}$, 
$x_{i+5}$, ..., $x_{i-2}$, $v_{i+4}$, $v_{i+5}$,..., and $v_{i-1}$ in $G^2$. Call this list $L_3$. We have $|L_3(v_i)| \ge 2$ and $|L_3(v_{i+1})| \ge 1$. Greedily color $v_{i+1}$ then $v_i$ in $G^2$ to obtain $\chi_l(G^2) \leq 11$, a contradiction. 
\item \textbf{Case 2:} There exists $c_2 \in L_1(x_i)$  such that $|L_1(v_{i}) \backslash \{c_2\}| \ge 5$. \\
Then, color $x_i$
by $c_2$ and call the new list $L_2$.
Since $g(G) >6$, we have $d(x_{i+2}, v_{i-1}) \ge 3$. We have $|L_2(v_{i+1})|\ge 4$, $|L_2(v_{i-1})|\ge 4$, and $L_2(x_{i+2})|\ge2 $.
Since $|L_2(x_{i+2})| + |L_2(v_{i-1})| > 5 $, we can color $x_{i+2}$ by $c_3$ and $v_{i-1}$ by $c_4$ such that $|L_2(v_{i+1}) \backslash \{c_3,c_4\}|\ge 3$. Color $x_{i+2}$ by $c_3$ and $v_{i-1}$ by $c_4$ then greedily color $x_{i+2}$, $x_{i+3}$, ..., $x_{i-2}$, $v_{i-2}$, $v_{i-3}$,..., and $v_{i+2}$ which is possible since each vertex has at least two uncolored neighbors in $G^2$. Call this list $L_3$. We have $|L_3(v_i)| \ge 2$ and $|L_3(v_{i+1})| \ge 1$. Greedily color $v_{i+1}$ then $v_i$ in $G^2$ to obtain $\chi_l(G^2) \leq 11$, a contradiction. 
\end{itemize}
\end{proof}

\textbf{Final Step:} Recall that $|L(x_i)|\ge 2$ for all $i$, $1\leq i \leq k$. Let $c_i, c_i' \in L(x_i)$ such that $c_i\neq c_{i}'$. By \Cref{c1}, we deduce that $c_i, c_i' \in L(v_i)$. Color each $x_i$ by $c_i$ and each $v_i$ by $c_i'$. Recall that $d(v_i,v_j) \ge 3 $ for $j \notin \{i-2,i-1,i+1,i+2\}$ by \Cref{l10}. Since
$L(x_i)\cap L(x_{i+1}) =L(x_{i-1})\cap L(x_{i+1})= \emptyset$ by \Cref{c2} and \ref{c3}, we deduce that $c_i \neq c_j$
for $j \in \{i-2, i-1, i+1 , i+2\}$. Thus, the coloring in $G^2$ is proper
and so we deduce that $\chi_l(G^2) \leq 11$, a contradiction. Hence, $G$ does not exist. Therefore, for every $4$-irregular graph $G$, we have $\chi_l(G^2)\leq 11$.

\section{Open Problems}
\begin{problem}
    Let $G$ be a graph such that $\Delta(G)\leq 4$ and each $4$-vertex has at most one $4$-neighbor. Is $\chi_l(G^2)\leq 11?$
\end{problem}
\begin{problem}
    Let $G$ be a graph such that $\Delta(G)\leq 4$ and each $4$-vertex has at most $k$ $4$-neighbors, $k\in \{2,3\}$. Is $\chi_l(G^2)\leq 10+k?$
\end{problem}

\vspace{0.5cm}
\textbf{Acknowledgment:} I would like to express my heartfelt appreciation to Dr. Maidoun Mortada for her continuous support and guidance in completing this work.
\end{large}


\begin{thebibliography} {99}
\bibitem{1} G. Agnarsson, M. Halldórsson, Coloring powers of planar graphs, 
\textit{SIAM Journal on Discrete Mathematics} 16(4) (2003), 651-662.
\bibitem{6} K. Aoki, Improved 2-distance coloring of planar graphs with maximum degree five, \textit{Discrete Mathematics} 348(1) (2025), 114225.
\bibitem{2} B. Borodin, H. Glebov, Stars and Bunches in Planar Graphs. Part II: General Planar Graphs and Colourings, University of Twente, 2002,
2002-05.
\bibitem{3} N. Bousquet, W. Deschamps, L. De Meyer, T. Pierron, Improved square coloring of
planar graphs, \textit{Discrete Mathematics} 346(4) (2023), 113288.
\bibitem{4} M. Chen, L. Miao, S. Zhou, 2-Distance coloring of planar graphs
with maximum degree 5, \textit{Discrete Mathematics} 345 (2022), 112766.
\bibitem{ck} D. W. Cranston and S.-J. Kim, List-coloring the Square of a Subcubic Graph, \textit{Journal of Graph Theory}, 57(1) (2008), 65-87.
\bibitem{cs} D. W. Cranston and R. \v{S}krekovski, Sufficient sparseness conditions for $G^2$ to be $(\Delta+1)$-choosable, when $\Delta \ge 5$. \textit{Discrete Appl. Math.,} 162 (2014),167-176.
\bibitem{5} Z. Deniz, On 2-distance 16-coloring of planar graphs with maximum degree at most five, \textit{Discrete Mathematics} 348(4) (2025), 114379.
\bibitem{ds} M.H. Dolama and E. Sopena, On the maximum average degree and the incidence chromatic number of a graph, \textit{Discrete Math. and Theoret. Comput. Sci.}, 7(1) (2005), 203-216.
\bibitem{jkk} L. Jin, Y. Kang, and S.-J. Kim, The square of every subcubic planar graph without 4-cycles and 5-cycles is 7-choosable, \textit{Graphs and Combinatorics}, 42, (2026).
\bibitem{klno} S.-J. Kim, X. Lian, A. Nakamoto, and K. Ozeki, The square of a subcubic planar graph without a 5-cycle is 7-choosable, (2025) arXiv.2512.24536.  
\bibitem{kl} S.-J. Kim and X. Lian, The square of every subcubic planar graph of girth at least
6 is 7-choosable, \textit{ Discrete Mathematics}, 347(6) (2024),113963.
\bibitem{7} M. Krzyzinski, P. Rzazewski, and S. Tur. Coloring squares of planar graphs with small maximum degree, 2021, arXiv:2105.11235v1.

\bibitem{8} M. Molloy, M.R. Salavatipour, A bound on the chromatic number
of the square of a planar graph, \textit{Journal of Combinatorial Theory, Series B} 94(2) (2005), 189-213.
\bibitem{9} C. Thomassen, The square of a planar cubic graph is 7-colorable, \textit{Journal of Combinatorial Theory, Series B}, 128 (2018), 192–218.
\bibitem{10} J. Van den Heuvel, S. Mcguinness, Coloring the square of a planar graph, \textit{Journal of Graph Theory} 42(2) (2003), 42-110.
\bibitem{11} G. Wegner, Graphs with given diameter and a coloring problem, University of Dortmund, 1977.
\bibitem{12} J. Zhu, Y. Bu, Minimum 2-distance coloring of planar graphs and
channel assignment, \textit{Journal of Combinatorial Optimization} 36(1) (2018), 55-64.
\bibitem{13} J. Zou, M. Han, H. Lai, Square coloring of planar graphs
with maximum degree at most five, \textit{Discrete Applied Mathematics} 353 (2024), 4-11.

\end{thebibliography}
\end{document}